\documentclass[11pt]{amsart}
\usepackage{extarrows}
\usepackage[colorlinks, citecolor=blue, dvipdfm, pagebackref]{hyperref}

\newtheorem{theorem}{Theorem}[section]
\newtheorem{lemma}[theorem]{Lemma}

\theoremstyle{definition}
\newtheorem{definition}[theorem]{Definition}

\newtheorem{remark}[theorem]{Remark}

\theoremstyle{remark}

\newcommand{\be}{\begin{equation}}
\newcommand{\ee}{\end{equation}}

\numberwithin{equation}{section}

\begin{document}

\title{Vanishing theorems on compact Chern-K\"{a}hler-like Hermitian manifolds}

\author{Ping Li}
\address{School of Mathematical Sciences, Fudan University, Shanghai 200433, China}
\address{School of Mathematical Sciences, Tongji University, Shanghai 200092, China}

\email{pingli@tongji.edu.cn\\
pinglimath@gmail.com}
\thanks{The author was partially supported by the National
Natural Science Foundation of China (Grant No. 11722109) and the Fundamental Research Funds for the Central Universities.}

%    General info

%\date{January 1, 2001 and, in revised form, June 22, 2001.}

%\dedicatory{}

\begin{abstract}
We show that, under the definiteness of holomorphic sectional curvature, the spaces of some holomorphic tensor fields on compact Chern-K\"{a}hler-like Hermitian manifolds are trivial. These can be viewed as counterparts to Bochner's classical vanishing theorems on compact K\"{a}hler manifolds under the definiteness of Ricci curvature or the existence of K\"{a}hler-Einstein metrics. Some arguments in our proof are inspired by and based on some ideas due to X.K. Yang and L. Ni-F.Y. Zheng.
\end{abstract}

\maketitle

%\tableofcontents

\section{Introduction and main results}\label{introduction}
Throughout this article denote by $(M,\omega)$ a compact connected complex manifold of complex dimension $n\geq2$ endowed with a Hermitian metric whose associated positive $(1,1)$-form is $\omega$. By abuse of notation, $\omega$ itself is also called the Hermitian metric. Let $TM$ and $T^{\ast}M$ be the holomorphic tangent and cotangent bundles of $M$ respectively, and
$$\Gamma^p_q(M):=H^0\big(M,(TM)^{\otimes p}\otimes(T^{\ast}M)^{\otimes q}\big),\qquad (p, q\in\mathbb{Z}_{\geq 0})$$
the space of $(p,q)$-type \emph{holomorphic} tensor fields on $M$.

When $(M,\omega)$ is K\"{a}hler, Bochner noticed that (see \cite[\S 8]{YB}) definiteness properties of the Ricci curvature or existence of K\"{a}hler-Einstein metrics impose heavy restrictions on $\Gamma^p_q(M)$. The main idea of the proof is, for $T\in\Gamma^p_q(M)$, the eigenvalues of the Ricci curvature are involved in the difference $\Delta|T|^2-|\nabla T|^2$, where $\Delta(\cdot)$ is the Laplacian operator and $\nabla(\cdot)$ the Levi-Civita connection of $\omega$. This trick and its later various variants are called \emph{the Bochner techniques} and have far-reaching impacts in and beyond differential geometry (see \cite{Wu88}). Thanks to the later celebrated Calabi-Yau theorem and Aubin-Yau theorem (\cite{Yau77}), (Some of) Bochner's results can now be reformulated in terms of definiteness of the first Chern class $c_1(M)$ as follows (\cite{Ko}, \cite{Ko1}, \cite[p. 57]{KH}).
\begin{theorem}[Bochner, Calabi-Yau, Aubin-Yau]\label{Bochner}
Let $(M,\omega)$ be a compact K\"{a}hler manifold.
\begin{enumerate}
\item
If $c_1(M)$ is quasi-positive, then $\Gamma^p_q(M)=0$ when $q>>p$. In particular $\Gamma^0_q(M)=0$ when $q\geq 1$ and consequently the Hodge numbers $h^{q,0}(M)=0$ when $q\geq 1$.

\item
If $c_1(M)$ is quasi-negative, then $\Gamma^p_q(M)=0$ when $p>>q$. In particular $\Gamma^p_0(M)=0$ when $p\geq 1$.

\item
If $c_1(M)<0$, then $\Gamma^p_q(M)=0$ when $p>q$.

%\item
%If $\text{Ric}(\omega)=\omega$, i.e., $(M,\omega)$ is a Fano K\"{a}hler-Einstein manifold,  then $\Gamma^p_q(M)=0$ when $q>p$.
\end{enumerate}
\end{theorem}
\begin{remark}
\begin{enumerate}
\item
Quasi-positivity (resp. quasi-negativity) of $c_1(M)$ means that there exists a closed $(1,1)$-form representing $c_1(M)$ which is nonnegative (resp. non-positive) everywhere and positive (resp. negative) somewhere.
\item
The qualitative conditions ``$q>>p$" and ``$p>>q$" can be made precise in terms of the maximum and minimum of the eigenvalues of the Ricci curvature. Details can be found in \cite[\S 8]{YB} and \cite{KH}.
\item
Another early usage of the Bochner technique to show holomorphic vanishing results is due to Kobayashi and H. Wu in \cite{KW}, where the curvature conditions are imposed on the Ricci curvature of the holomorphic vector bundles.
\end{enumerate}
\end{remark}

For K\"{a}hler manifolds the relationship between Ricci curvature (``Ric" for short) and holomorphic sectional curvature (``HSC" for short) is quite subtle and even mysterious. On one hand, both $\text{Ric}(\omega)$ and $\text{HSC}(\omega)$ dominate and are dominated by scalar curvature and holomorphic bisectional curvature respectively. On the other hand, in general $\text{Ric}(\omega)$ and $\text{HSC}(\omega)$ do not dominate each other. Nevertheless, a recent breakthrough due to Wu-Yau, Tosatti-Yang and Diverio-Trapani (\cite{WY16-1}, \cite{TY17}, \cite{DT19}, \cite{WY16-2}) tells us that the quasi-negativity of HSC implies $c_1<0$. As a consequence, the conclusions of Parts $(2)$ and $(3)$ in Theorem \ref{Bochner} remain true under the assumption of $HSC(\omega)$ being quasi-negative. In view of this, one may wonder whether similar statement holds true for compact K\"{a}hler manifolds with positive HSC. However, it turns out that (\cite{Hi75}, \cite[p. 949]{Yang16}) the Hirzebruch surfaces $$M=\mathbb{P}\big(\mathcal{O}_{\mathbb{P}^1}(-k)
\oplus\mathcal{O}_{\mathbb{P}^1}\big) \qquad(k\geq 2) $$
admit K\"{a}hler metrics with positive HSC but $c_1(M)$ are \emph{not} positive.

Even if in general the positivity of HSC cannot imply that of the first Chern class, we may still ask if the conclusions of Part $(1)$ in Theorem \ref{Bochner} hold true under the assumption of HSC being positive. In his recent work \cite{Yang18}, X.K. Yang proved a conjecture of S.-T. Yau (\cite[Problem 47]{Yau82}), which states that a compact K\"{a}hler manifold $(M,\omega)$ with $\text{HSC}(\omega)>0$ must be projective  and thus gives a metric criterion of the projectivity. Indeed Yang showed that the condition of $\text{HSC}(\omega)>0$ leads to the Hodge numbers $h^{q,0}(M)=0$ for $1\leq q\leq n$ via the notion of \emph{RC-positivity} introduced by himself (\cite[Thm 1.7]{Yang18}). In particular $h^{2,0}(M)=0$ implies the projectivity of $M$ due to a well-known result of Kodaira (\cite[p. 143]{MK71}). Note that $\Gamma^0_q(M)=0$ yields $h^{q,0}(M)=0$ and hence the vanishing results Yang
obtained provide evidence towards the expected conclusions of Part $(1)$ in Theorem \ref{Bochner}.

In another recent work \cite{NZ}, L. Ni and F.Y. Zheng proved that the projectivity remains true under a weaker positivity: the $2$nd scalar curvature $S_2(\omega)$. they indeed introduced in \cite{NZ} a family of $S_k(\omega)$ ($1\leq k\leq n$), called $k$-th sclar curvatures, by averaging $\text{HSC}(\omega)$ over $k$-dimensional subpaces of the tangent spaces, and the positivity of $\text{HSC}(\omega)$ implies that of $S_2(\omega)$, and showed that (\cite[Thm 1.1]{NZ}) the Hodge numbers $h^{q,0}(M)=0$ for $2\leq q\leq n$ whenever $S_2(\omega)>0$.

\emph{The purpose} of this note is, by building on some ideas in \cite{Yang18} and \cite{NZ}, to show that the conclusions of Part $(1)$ in Theorem \ref{Bochner} still hold true under the assumption of HSC being positive. In fact, we will show them in a more general setting. For this purpose, let us recall the following notion, which was introduced and investigated by B. Yang and F. Zheng in \cite{YZ}.

\begin{definition}
Let $(M,\omega)$ be a Hermitian manifold and $R$ the curvature tensor of the \emph{Chern connection}, which is the unique canonical connection compatible with both the metric and the complex structure. The Hermitian metric $\omega$ is called \emph{Chern-K\"{a}hler-like} (CKL for short) if \be\label{CKL}
R(X,\overline{Y},Z,\overline{W})=R(Z,\overline{Y},X,\overline{W})\ee for any $(1,0)$-type tangent vectors $X$, $Y$, $Z$ and $W$.
\end{definition}
\begin{remark}\label{remark}
When $\omega$ is K\"{a}hler, $R$ is the usual curvature tensor, which clearly satisfies (\ref{CKL}).
Note that, by taking complex conjugations, $(\ref{CKL})$ implies that
$$R(X,\overline{Y},Z,\overline{W})=
\overline{R(Y,\overline{X},W,\overline{Z})}=
\overline{R(W,\overline{X},Y,\overline{Z})}=
R(X,\overline{W},Z,\overline{Y})
.$$
Therefore the condition (\ref{CKL}) ensures that $R$ obeys \emph{all} the symmetries satisfied by the curvature tensor of a K\"{a}hler metric and thus the term CKL is justified.
\end{remark}
As pointed out in \cite[p. 1197]{YZ}, there are plenty of non-K\"{a}hler Hermitian metrics which are CKL. On the other hand, it turns out that (\cite[Thm 3]{YZ}) a CKL Hermitian metric $\omega$ must be balanced, i.e., $d\omega^{n-1}=0$. Hence CKL Hermitian metrics interpolate between K\"{a}hler and balanced metrics.

For a Hermitian manifold $(M,\omega)$, $x\in M$ and $v\in T_xM-\{0\}$,
the \emph{holomorphic sectional curvature} of $\omega$, denoted by $\text{HSC}(\omega)$, at the point $x$ and the direction $v$ is defined by $$H_x(v):=\frac{R(v,\bar{v},v,\bar{v})}{|v|^4}.$$ $\text{HSC}(\omega)$ is called positive, denoted by $\text{HSC}(\omega)>0$, if
$H_x(v)>0$ for any $x\in M$ and $v\in T_xM-\{0\}$. Note that this $H_x(\cdot)$ is indeed defined on $\mathbb{P}(T_xM)$, the projectivation of $T_xM$. Thus the maximal and minimal values of $H_x$ can be attained, a fact which will be used later.

With these notions in mind, our main result in this note is the following theorem, which can be viewed as counterparts to Theorem \ref{Bochner}.
\begin{theorem}\label{CKL result}
Let $(M,\omega)$ be a compact CKL Hermitian manifold.
\begin{enumerate}
\item
If $\text{HSC}(\omega)>0$, then $\Gamma^p_q(M)=0$ when $q>>p$. In particular $\Gamma^0_q(M)=0$ when $q\geq 1$ and consequently the Hodge numbers $h^{q,0}(M)=0$ when $1\leq q\leq n$.

\item
If $\text{HSC}(\omega)<0$, then $\Gamma^p_q(M)=0$ when $p>>q$. In particular $\Gamma^p_0(M)=0$ when $p\geq 1$.
\end{enumerate}
\end{theorem}

\begin{remark}
\begin{enumerate}
\item
As remarked above, when the metric $\omega$ is K\"{a}hler, Part $(2)$ in Theorem \ref{CKL result} follows from Theorem \ref{Bochner} and the recent result due to Wu-Yau et. al. However, if the CKL metric $\omega$ is \emph{non}-K\"{a}hler, the statements in Part $(2)$ are also new.

\item
The qualitative conditions ``$q>>p$" and ``$p>>q$" in Theorem \ref{CKL result} can be made more quantitative. Details can be found in Theorem \ref{quantitative version} and its proof.
\end{enumerate}
\end{remark}

\section{Preliminaries}\label{preliminaries}
We briefly collect in this section some basic facts on Hermitian holomorphic vector bundles and Hermitian manifolds in the form we shall use to prove Theorem \ref{CKL result}. A thorough treatment can be found in \cite{Ko2}.

Let $(E^r,h)\rightarrow M$ be a Hermitian holomorphic vector bundle of rank $r$ on an $n$-dimensional compact complex manifold $M$
endowed with the canonical Chern connection $\nabla$ and the curvature tensor $$R:=\nabla^2\in\Gamma(\Lambda^{1,1}M\otimes E^{\ast}\otimes E).$$
Here and throughout this section we use $\Gamma(\cdot)$ to denote the space of \emph{smooth} sections for vector bundles and the notation $\Gamma^p_q(M)$ is reserved to denote that of the \emph{holomorphic} $(p,q)$-tensor fields on $M$ as before.

Under a local frame field $\{s_1,\ldots,s_r\}$ of $E$, whose dual coframe field is denoted by $\{s_1^{\ast},\ldots,s_r^{\ast}\}$, and local coordinates $\{z^1,\ldots,z^n\}$ on $M$, the curvature tensor $R$ and the Hermitian metric $h$ can be written as
\begin{eqnarray}\label{curvature tensor}
\left\{ \begin{array}{ll}
R=:\Omega^{\beta}_{\alpha}s^{\ast}_{\alpha}\otimes s_{\beta}=:R^{\beta}_{i\bar{j}\alpha}dz^i\wedge d\bar{z}^j\otimes s^{\ast}_{\alpha}\otimes s_{\beta},\\
~\\
h=(h_{\alpha\bar{\beta}}):=\big(h(s_{\alpha},s_{\beta})\big),\\
~\\
R_{ij\alpha\bar{\beta}}:=R_{ij\alpha}^{\gamma}h_{\gamma\bar{\beta}}.
\end{array} \right.
\end{eqnarray}
Here and in what follows we always adopt the Einstein summation convention. For simplicity we sometimes use $<\cdot,\cdot>$ to denote the Hermitian metric $h(\cdot,\cdot)$ and the induced metrics on various vector bundles arising naturally from $E$.

Recall that for $\eta=\eta^{\alpha}s_{\alpha}\in\Gamma(E)$,
$$R(\eta)=(\Omega^{\beta}_{\alpha}s^{\ast}_{\alpha}\otimes s_{\beta})(\eta^{\gamma}s_{\gamma})=\Omega^{\beta}_{\alpha}\eta^{\alpha}s_{\beta},$$
and $u\in u^i\frac{\partial}{\partial z^i}$, $v\in v^i\frac{\partial}{\partial z^i}$,
$$R_{u\bar{v}}(\eta)=\Omega^{\beta}_{\alpha}
(u,\bar{v})\eta^{\alpha}s_{\beta}=R^{\beta}_{i\bar{j}\alpha}u^i\bar{v^j}\eta^{\alpha}s_{\beta},$$
and thus
\be\label{Hermitian formula}
\begin{split}
<R_{u\bar{v}}(\eta),\xi>=&
<R^{\beta}_{i\bar{j}\alpha}u^i\bar{v^j}\eta^{\alpha}s_{\beta},\xi^{\gamma}s_{\gamma}>\\
=&R^{\beta}_{i\bar{j}\alpha}u^i\bar{v^j}\eta^{\alpha}\bar{\xi^{\gamma}}h_{\beta\bar{\gamma}}\\
=&R_{i\bar{j}\alpha\bar{\beta}}u^i\bar{v^j}\eta^{\alpha}\bar{\xi^{\beta}}.\end{split}
\ee

A direct consequence of (\ref{Hermitian formula}) is that $R_{u\bar{u}}(\cdot)$ is a \emph{Hermitian transformation}: \be\label{Hermitan transformation}<R_{u\bar{u}}(\eta),\xi>=<\eta,R_{u\bar{u}}(\xi)>,\qquad\text{for all $\eta,\xi\in\Gamma(E).$}\ee

The following two lemmas are crucial to our proof.
\begin{lemma}\label{Bochner fromula0 lemma}
Let $\eta\in\Gamma(E)$ be holomorphic and the maximum of $|\eta|:=<\eta,\eta>^{\frac12}$ is attained at $x\in M$. Then
\be\label{Bochner formula0}<R_{u\bar{u}}(\eta),\eta>\big|_x\geq0,\qquad\text{for all $u\in T_xM.$}\ee
\end{lemma}
\begin{remark}
In terms of local coordinates, (\ref{Bochner formula0}) is equivalent to the fact that $$\big(R_{i\bar{j}\alpha\bar{\beta}}\eta^{\alpha}\bar{\eta^{\beta}}\big)\big|_x\geq0$$
as a Hermitian matrix.
\end{remark}
\begin{proof}
A well-known Bochner-type formula reads (\cite[p.50]{Ko2})
\be\label{Bochner formula}\partial\bar{\partial}|\eta|^2=<\nabla\eta,\nabla\eta>-<R(\eta),\eta>,\ee
where $<R(\eta),\eta>$ is understood to pair the elements in $\Gamma(E)$ and maintain those in $\Gamma(\Lambda^{1,1}M)$. The holomorphicity of $\eta$ implies that $\nabla\eta\in\Gamma(\Lambda^{1,0}M\otimes E)$ and hence $$<\nabla\eta,\nabla\eta>\in\Gamma(\Lambda^{1,1}M)$$
is similarly understood.
The reader is referred to \cite[p.50]{Ko2} for (\ref{Bochner formula}) in terms of local coordinates.

For any vector $u\in T_xM$ we apply $u\wedge\bar{u}$ to evaluate both sides of (\ref{Bochner formula}) to yield
\be\label{Bochner formula2}<R_{u\bar{u}}(\eta),\eta>\big|_x=
\big(-\sqrt{-1}\partial\bar{\partial}|\eta|^2\big)\Big|_x
\big(\frac{1}{\sqrt{-1}}u\wedge\bar{u}\big)+|\nabla_u\eta|^2.\nonumber\ee
Since the maximum of $|\eta|$ is attained at $x$, the maximum principle implies that $$\big(-\sqrt{-1}\partial\bar{\partial}|\eta|^2\big)\Big|_x\geq0$$ as a $(1,1)$-form, and hence we arrive at the desired result.
\end{proof}
\begin{remark}
The above proof of using $\partial\bar{\partial}$-Bochner formula and applying the maximum principle to part of directions are indeed inspired by some arguments in \cite[Prop. 4.2]{Yang18} and \cite[Lemma 2.1]{NZ}. Similar techniques and ideas can also be found in \cite{An}, \cite{AC}, \cite{Liu} and \cite{Ni}.
\end{remark}

The following lemma is parallel to \cite[Lemma 6.1]{Yang18}, where the conclusion was stated for K\"{a}hler manifolds.
\begin{lemma}\label{inequality}
Let $(M,\omega)$ be a compact CKL Hermitian manifold, $x\in M$, and the unit vector $u\in T_xM$ minimizes (resp. maximizes) the holomorphic sectional curvature at $x$. Then
\be\label{formula inequality} R(u,\bar{u},v,\bar{v})\underset{(\leq)}{\geq}\frac{1+|<u,v>|^2}{2}H_x(u),\qquad\text{for all unit vectors $v\in T_xM$}.\ee
\end{lemma}
\begin{remark}
As pointed out in \cite[p.198]{Yang18}, the tricks and various variants used in the proof of Lemma \ref{inequality} can be found in \cite[p.312]{Go98}, \cite[p.136]{Br}, \cite[Lemma 1.4]{BYT} and \cite[Lemma 4.1]{Yang17}. We remark that similar tricks can also be found in \cite{BG} and \cite[\S2]{Gr}, to the author's best knowledge. Although \cite[Lemma 6.1]{Yang18} is stated for K\"{a}hler manifolds, we will see in the course of the proof below that what it really needs is various K\"{a}hler-type symmetries of the curvature tensor $R$, which are satisfied by CKL Hermitian metrics as explained in Remark \ref{remark}. For the sake of the reader's convenience as well as completeness, we still include a proof below.
\end{remark}
\begin{proof}
Assume that a unit vector $w\in T_xM$ is such that $<u,w>=0$. Note that under this assumption we have
$$\big|(\cos\theta)u+(\sin\theta)w\big|=\big|(\cos\theta)u+(\sqrt{-1}\sin\theta) w\big|=1,\qquad\text{for all $\theta\in\mathbb{R}$.}$$
We consider
\be\begin{split}
f(\theta):=&H_x\big((\cos\theta) u+(\sin\theta) w\big)\\
=&R\big((\cos\theta)u+(\sin\theta) w,\ldots,\overline{(\cos\theta) u+(\sin\theta) w}\big),\end{split}\nonumber\ee
and
\be\begin{split}
g(\theta):=&H_x\big((\cos\theta) u+(\sqrt{-1}\sin\theta) w\big)\\
=&R\big((\cos\theta) u+(\sqrt{-1}\sin\theta) w,\ldots,\overline{(\cos\theta) u+(\sqrt{-1}\sin\theta) w}\big).
\end{split}\nonumber\ee

Since $u$ minimizes (resp. maximizes) $H_x(\cdot)$ and $f(0)=g(0)=H_x(u)$, $\theta=0$ attains the minimal (resp. maximal) value of $f(\theta)$ and $g(\theta)$. This means that
\be\label{1}f'(0)=g'(0)=0\ee
and
\be\label{1.5}
\text{$f''(0)\underset{(\leq)}{\geq}0$,~ $g''(0)\underset{(\leq)}{\geq}0$}.\ee

Denote by $R_{1\bar11\bar1}:=R(u,\bar{u},u,\bar{u})$, $R_{1\bar12\bar1}:=R(u,\bar{u},w,\bar{u})$ and so on. Direct calculations, together with various K\"{a}hler-like symmetries of $R_{i\bar{j}k\bar{l}}$ ensured by the CKL condition, show that
\be\label{2}f'(0)=R_{1\bar11\bar2}+R_{1\bar12\bar1},\qquad g'(0)=-R_{1\bar11\bar2}+R_{1\bar12\bar1}\ee
and
\begin{eqnarray}\label{3}
\left\{ \begin{array}{ll}
f''(0)=4R_{1\bar12\bar2}+R_{1\bar21\bar2}+R_{2\bar12\bar1}-2R_{1\bar11\bar1}\\
~\\
g''(0)=4R_{1\bar12\bar2}-R_{1\bar21\bar2}-R_{2\bar12\bar1}-2R_{1\bar11\bar1}.
\end{array} \right.
\end{eqnarray}
Combining (\ref{1}) with (\ref{2}) implies that
\be\label{4}R_{1\bar11\bar2}=R_{1\bar12\bar1}=0,\ee
and (\ref{1.5}) with (\ref{3}) leads to
\be\label{5}2R_{1\bar12\bar2}\underset{(\leq)}{\geq} R_{1\bar11\bar1}.\ee

Now for any unit vector $v\in T_xM$, we can choose a unit vector $w\in T_xM$ such that $<u,w>=0$ and $v=au+bw$ with $|a|^2+|b|^2=1$. Then
\be\begin{split}
R(u,\bar{u},v,\bar{v})&=R(u,\bar{u},au+bw,\overline{au+bw})\\
&=|a|^2R_{1\bar11\bar1}+|b|^2R_{1\bar12\bar2}
\qquad\big(\text{by (\ref{4})}\big)\\
&\underset{(\leq)}{\geq}(|a|^2+\frac{|b|^2}{2})R_{1\bar11\bar1}
\qquad\big(\text{by (\ref{5})}\big)\\
&=\frac{1+|<u,v>|^2}{2}H_x(u),
\end{split}\nonumber\ee
which yields the desired inequality (\ref{formula inequality}).
\end{proof}

\section{Proof of Theorem \ref{CKL result}}\label{proof}
We can now proceed to prove the main result in this note, Theorem \ref{CKL result}.

Assume temporarily that $(M^n,\omega)$ is a general compact Hermitian manifold and $T\in\Gamma^p_q(M)$ a $(p,q)$-type holomorphic tensor field on it. Let $x\in M$ and \emph{unit} vector $u\in T_xM$. Let $\{e_1,\ldots,e_n\}$ be an orthonormal basis of $T_xM$ and $\{\theta^1,\ldots,\theta^n\}$ an orthonormal basis of $T_x^{\ast}M$ dual to $\{e_i\}$. Denote by
$$T\big|_x=T^{\alpha_1\cdots\alpha_p}_{\beta_1\cdots\beta_q}e_{\alpha_1}
\otimes\cdots\otimes e_{\alpha_p}\otimes\theta^{\beta_1}\otimes\cdots\otimes\theta^{\beta_q}.$$

Recall from (\ref{Hermitan transformation}) that $R_{u\bar{u}}(\cdot)$ is a Hermitian transformation and hence its eigenvalues are all real, say $\lambda_i=\lambda_i(x,u)$ ($1\leq i\leq n$). Without loss of generality we may choose the above orthonormal basis $\{e_1,\ldots,e_n\}$ such that
\be\label{eigenvalue}R_{u\bar{u}}(e_i)=\lambda_ie_i,\qquad\lambda_i\in\mathbb{R},\qquad 1\leq i\leq n.\ee
Note that the induced actions of $R_{u\bar{u}}$ on $\theta^1,\ldots,\theta^n$ are given by
$$R_{u\bar{u}}(\theta^i)=-\lambda_i\theta^i,\qquad 1\leq i\leq n,$$
and hence
$$R_{u\bar{u}}(T)\big|_x=
\sum_{\overset{\alpha_1,\ldots,\alpha_p}{\beta_1,\ldots,\beta_q}}
\Big[\big(\sum_{i=1}^p\lambda_{\alpha_i}-\sum_{j=1}^q\lambda_{\beta_j}\big)
T^{\alpha_1\cdots\alpha_p}_{\beta_1\cdots\beta_q}e_{\alpha_1}
\otimes\cdots\otimes e_{\alpha_p}\otimes\theta^{\beta_1}\otimes\cdots\otimes\theta^{\beta_q}\Big].$$
This yields that
\be\label{T-formula}<R_{u\bar{u}}(T)\big|_x,T\big|_x>=\sum_{\overset{\alpha_1,\ldots,\alpha_p}{\beta_1,\ldots,\beta_q}}
\big(\sum_{i=1}^p\lambda_{\alpha_i}-\sum_{j=1}^q\lambda_{\beta_j}\big)
\big|T^{\alpha_1\cdots\alpha_p}_{\beta_1\cdots\beta_q}\big|^2.\ee

Let
\be\label{pointmaxmin}\lambda_{\max}(x,u):=\max\big\{\lambda_i(x,u)\big\},\qquad
\lambda_{\min}(x,u):=\min\big\{\lambda_i(x,u)\big\}
\ee
and
\be\label{maxmin}\lambda_{\max}:=\max_{\overset{x\in M, u\in T_xM}{ |u|=1}}\lambda_{\max}(x,u),\qquad
\lambda_{\min}:=\min_{\overset{x\in M, u\in T_xM}{ |u|=1}}\lambda_{\min}(x,u)
.\ee
We remark that $\lambda_{\max}(x,u)$ and $\lambda_{\min}(x,u)$ are in general \emph{continuous} functions and may not be smooth. Nevertheless, continuity is enough to guarantee that $\lambda_{\max}$ and $\lambda_{\min}$ can both be \emph{attained} as the maximum and minimum in (\ref{maxmin}) are over the unit sphere bundle of $TM$, which is compact.

Similarly define
\be\label{maxminholo}H_{\max}:=\max_{\overset{x\in M, u\in T_xM}{ |u|=1}}H_x(u),\qquad
H_{\min}:=\min_{\overset{x\in M, u\in T_xM}{ |u|=1}}H_x(u).
\ee
For the same reason as above $H_{\max}$ and $H_{\min}$ can also be attained as they are smooth. Thus the positivity (resp. negativity) of $\text{HSC}(\omega)$ implies that of $H_{\min}$ (resp. $H_{\max}$).

Now we are ready to prove Theorem \ref{CKL result} in a more quantitative version, that is
\begin{theorem}\label{quantitative version}
Let $(M,\omega)$ be a compact CKL Hermitian manifold.
\begin{enumerate}
\item
If $\text{HSC}(\omega)>0$, then $\Gamma^p_q(M)=0$ whenever $2p\lambda_{\max}<qH_{\min}$. In particular $\Gamma^0_q(M)=0$ whenever $q\geq 1$ and consequently the Hodge numbers $h^{q,0}(M)=0$ for $1\leq q\leq n$.

\item
If $\text{HSC}(\omega)<0$, then $\Gamma^p_q(M)=0$ whenever $pH_{\max}<2q\lambda_{\min}$. In particular $\Gamma^p_0(M)=0$ whenever $p\geq 1$.
\end{enumerate}
\end{theorem}
\begin{proof}
Assume that $\text{HSC}(\omega)>0$. Let $T\in\Gamma^p_q(M)$, the maximal value of $|T|$ be attained at $x\in M$, and the unit vector $u\in T_xM$ minimize $H_x$.

First by the definitions (\ref{pointmaxmin}) and (\ref{maxmin}) the eigenvalues $\lambda_i(x,u)$ of the Hermitian transformation $R_{u\bar{u}}$ at $T_xM$ are bounded above by $\lambda_{\max}$: $\lambda_i(x,u)\leq\lambda_{\max}$.

Secondly, these eigenvalues $\lambda_i(x,u)$ are bounded below by $\frac{1}{2}H_{\min}$, which is positive under the assumption of $\text{HSC}(\omega)>0$. Indeed,
\be\begin{split}
\lambda_i(x,u)=&<R_{u\bar{u}}(e_i),e_i>
\qquad\big(\text{by (\ref{eigenvalue})}\big)\\
=&R(u,\bar{u},e_i,\bar{e_i})\\
\geq&\frac{1+|<u,e_i>|^2}{2}H_x(u)\qquad\big(\text{by (\ref{formula inequality})}\big)\\
\geq&\frac{1}{2}H_{\min}\qquad\big(\text{by (\ref{maxminholo})}\big)\\
>&0.
\end{split}\nonumber\ee

In summary, under the condition of $\text{HSC}(\omega)>0$, we have
\be\label{conclusion>0}\lambda_{\max}\geq\lambda_i(x,u)\geq\frac{1}{2}H_{\min}>0,\qquad\forall~1\leq i\leq n.\ee

Note that the key point here is that the two positive bounds $\lambda_{\max}$ and $\frac{1}{2}H_{\min}$ are \emph{independent} of the choice of the pair $(x,u)$ and hence depends only on the metric $\omega$. Therefore,
\be\label{6}
\begin{split}<R_{u\bar{u}}(T),T>\big|_x=&\sum_{\overset{\alpha_1,\ldots,\alpha_p}{\beta_1,\ldots,\beta_q}}
\big(\sum_{i=1}^p\lambda_{\alpha_i}-\sum_{j=1}^q\lambda_{\beta_j}\big)
\big|T^{\alpha_1\cdots\alpha_p}_{\beta_1\cdots\beta_q}\big|^2\qquad\big(\text{by (\ref{T-formula})}\big)\\
\leq&\sum_{\overset{\alpha_1,\ldots,\alpha_p}{\beta_1,\ldots,\beta_q}}
\big(p\cdot\lambda_{\max}-q\cdot\frac{1}{2}H_{\min}\big)
\big|T^{\alpha_1\cdots\alpha_p}_{\beta_1\cdots\beta_q}\big|^2\qquad\big(\text{by (\ref{conclusion>0})}\big)\\
\leq&0.\qquad \big(\text{when $2p\lambda_{\max}<qH_{\min}$}\big)
\end{split}\ee
However, (\ref{Bochner formula0}) in Lemma \ref{Bochner fromula0 lemma} implies that
$$<R_{u\bar{u}}(T),T>\big|_x\geq0,$$
which, together with (\ref{6}), in turn yields that $$<R_{u\bar{u}}(T),T>\big|_x=0.$$
Under the condition of the strict inequality $2p\lambda_{\max}<qH_{\min}$, the equality case of (\ref{6}) occurs only if all $|T^{\alpha_1\cdots\alpha_p}_{\beta_1\cdots\beta_q}\big|=0$ at $x$, i.e., $T(x)=0$. The maximum of $|T|$ at $x$ then lead to $T\equiv0$. This completes the first part in Theorem \ref{quantitative version}.

The proof of the second part is completely analogous. Under the condition of HSC being negative, we assume that the maximal value of $|T|$ be attained at $x$ and the unit vector $u\in T_xM$ \emph{maximizes} $H_x$.

Due to the same reasoning as in the first part we have
\be\label{conclusion<0}
\lambda_{\min}\leq\lambda_i(x,u)\leq\frac{1}{2}H_{\max}<0,\qquad(1\leq i\leq n)\nonumber\ee
and thus, when $pH_{\max}<2q\lambda_{\min}$,
\be\begin{split}
<R_{u\bar{u}}(T),T>\big|_x=&\sum_{\overset{\alpha_1,\ldots,\alpha_p}{\beta_1,\ldots,\beta_q}}
\big(\sum_{i=1}^p\lambda_{\alpha_i}-\sum_{j=1}^q\lambda_{\beta_j}\big)
\big|T^{\alpha_1\cdots\alpha_p}_{\beta_1\cdots\beta_q}\big|^2\\
\leq&\sum_{\overset{\alpha_1,\ldots,\alpha_p}{\beta_1,\ldots,\beta_q}}
\big(p\cdot\frac{1}{2}H_{\max}-q\cdot\lambda_{\min}\big)
\big|T^{\alpha_1\cdots\alpha_p}_{\beta_1\cdots\beta_q}\big|^2\\
\leq&0.\end{split}
\nonumber\ee
The same arguments as above again lead to $T\equiv0$, which completes the proof of the second part.
\end{proof}
\begin{remark}
We have seen from Theorem \ref{quantitative version} and its proof that the conditions ``$q>>p$" and ``$p>>q$" in Theorem \ref{CKL result} arise from $``2p\lambda_{\max}<qH_{\min}"$ and $``pH_{\max}<2q\lambda_{\min}"$ respectively. Although in the spirit of Theorem \ref{Bochner}, usually they are hard to verify, even if the metric $\omega$ is explicitly given.
\end{remark}

\section*{Acknowledgements}
The author would like to thank the referee for the careful reading and useful comments, and pointing out the reference \cite{KW}.

\end{document}